\documentclass{amsart}

\usepackage{amsmath}
\usepackage{amsthm}
\usepackage{amssymb}
\usepackage{amsfonts}
\usepackage{amsrefs}
\usepackage{color}
\usepackage{tikz}

\tikzstyle{shaded}=[fill=red!10!blue!20!gray!30!white]
\tikzstyle{shaded line}=[double=red!10!blue!20!gray!30!white, double distance=1.5mm, draw=black]
\tikzstyle{unshaded}=[fill=white]
\tikzstyle{unshaded line}=[double=white, double distance=1.5mm, draw=black]
\tikzstyle{Tbox}=[circle, draw, thick, fill=white, opaque,]
\tikzstyle{empty box}=[circle, draw, thick, fill=white, opaque, inner sep=2mm]
\tikzstyle{background rectangle}= [fill=red!10!blue!20!gray!40!white,rounded corners=2mm] 
\tikzstyle{on}=[very thick, red!50!blue!50!black]
\tikzstyle{off}=[gray]

\tikzstyle{traces}=[scale=.2, inner sep=1mm]
\tikzstyle{quadratic}=[scale=.4, inner sep=1mm, baseline]
\tikzstyle{annular}=[scale=.7, inner sep=1mm, baseline]
\tikzstyle{make triple edge size}= [scale=.4, inner sep=1mm,baseline] 
\tikzstyle{icosahedron network}=[scale=.3, inner sep=1mm, baseline]
\tikzstyle{ATLsix}=[scale=.25, baseline]
\tikzstyle{TL12}=[scale=.15,baseline]
\tikzstyle{PAdefn}=[scale=.7,baseline]
\tikzstyle{TLEG}=[scale=.5,baseline]

\newtheorem{lemma}{Lemma}[section]
\newtheorem{definition}[lemma]{Definition}
\newtheorem{theorem}[lemma]{Theorem}
\newtheorem{proposition}[lemma]{Proposition}
\newtheorem{remark}[lemma]{Remark}
\newtheorem{corollary}[lemma]{Corollary}
\newtheorem{conjecture}[lemma]{Conjecture}

\newtheorem{question}[lemma]{Question}

\newenvironment{claim}[1]{\par\noindent\underline{Claim:}\space#1}{}
\newenvironment{claimproof}[1]{\par\noindent\underline{Proof:}\space#1}{\hfill $\blacksquare$}

\DeclareMathOperator{\tr}{tr}
\DeclareMathOperator{\id}{id}
\DeclareMathOperator{\im}{im}

\makeatother

\title{Euler totient of subfactor planar algebras}
\author{S\'ebastien Palcoux}
\address{Institute of Mathematical Sciences, Chennai, India}
\email{sebastienpalcoux@gmail.com}

\begin{document}
\maketitle

\begin{abstract}
We extend the Euler's totient function (from arithmetic) to any irreducible subfactor planar algebra, using the M\"obius function of its biprojection lattice, as Hall did for the finite groups. We prove that if it is nonzero then there is a minimal $2$-box projection generating the identity biprojection. We explain a relation with a problem of K.S. Brown. As an application, we define the dual Euler totient of a finite group and we show that if it is nonzero then the group admits a faithful irreducible complex representation. We also get an analogous result at depth $2$, involving the central biprojection lattice.
\end{abstract}

\section{Introduction} 
Any finite group $G$ acts outerly on the hyperfinite ${\rm II}_1$ factor $R$, and the group subfactor $(R\subseteq R\rtimes G)$, of index $|G|$, remembers the group \cite{jo}. Jones proved in \cite{jo2} that the set of possible values for the index $ |M:N| $ of a subfactor $ (N \subseteq M) $ is 
 $$ \{ 4 cos^2(\frac{\pi}{n}) \ \vert \ n \geq 3 \} \sqcup [4,\infty]. $$
By Galois correspondence \cite{nk}, the lattice of intermediate subfactors of $(R\subseteq R\rtimes G)$ is isomorphic to the subgroup lattice of $G$. 
Moreover, Watatani \cite{wa} extended the finiteness of the subgroup lattice to any irreducible finite index subfactor. Then, the subfactor theory can be seen as an augmentation of the finite group theory, where the indices are not necessarily integers. The notion of subfactor planar algebra \cite{jo4} is a diagrammatic axiomatization of the standard invariant of a finite index ${\rm II}_1$ subfactor \cite{js}. Bisch \cite{bi} proved that the intermediate subfactors are given by the biprojections (see Definition \ref{debi}) in the $2$-box space of the corresponding planar algebra. The recent results of Liu \cite{li} on the biprojections are also crucial for this paper (see Section \ref{sectionB}).

The usual Euler's totient function $\varphi(n)$ counts the number of positive integers up to $n$ that are relatively prime to $n$. Let $G$ be a finite group and $\mu$ the M\"obius function (see Section \ref{sectionE}, Definition \ref{debius}) of its subgroup lattice $\mathcal{L}(G)$. Hall proved in \cite{hal} that the Euler totient of $G$ (defined below) is the cardinal of $\{g \in G \ | \ \langle g \rangle = G  \}$. $$\varphi(G):=\sum_{H \in \mathcal{L}(G)}\mu(H,G)|H|$$ So if $\varphi(G)$ is nonzero then $G$ is cyclic, and $\varphi(C_n) = \varphi(n)$. This implication will be generalized in Section \ref{main}, as the author did with Ore's theorem in \cites{p1,p2}. Let $\mathcal{P}$ be an irreducible subfactor planar algebra and $\mu$ the M\"obius function of its biprojection lattice $[e_1,\id]$. We use notations of Definition \ref{index}. Let the Euler totient of $\mathcal{P}$ be $$ \varphi(\mathcal{P}):=\sum_{b \in [e_1,\id]}\mu(b,\id)|b:e_1|. $$

\begin{theorem} \label{mainintro}
If $\varphi(\mathcal{P})$ is nonzero then $\mathcal{P}$ is w-cyclic (i.e. there is a minimal $2$-box projection generating the identity biprojection).
\end{theorem}

\noindent Then, for any finite group $G$, by considering $\mathcal{P}(R^G \subset R)$, we get that if the dual Euler totient  $$\hat{\varphi}(G):=\sum_{H \in \mathcal{L}(G)}\mu(1,H)|G:H|$$ is nonzero then $G$ has a faithful irreducible complex representation. It is a dual version of the initial implication. 
As a general application, we get a non-trivial upper bound for the minimal number of minimal central projections generating the identity biprojection. By applying this result to any finite group $G$, we deduce a non-trivial upper bound for the minimal number of irreducible components for a faithful complex representation of $G$. It involves the subgroup lattice and indices only. It is a link between combinatorics and representations in finite group theory. These results were not known to group theorists and the author can provide a group theoretical translation of the proofs as he did in \cite{p4} for the dual Ore's theorem. 

  This path of research conducts to (Section \ref{SectionK}) a generalization of a problem of K.S. Brown (considered hard in \cite{sw}) and to a possible counter-example of Watatani's problem \cite{wa} on whether any finite lattice has a biprojection lattice representation.

  We finally prove an additional result for the irreducible subfactor planar algebras of depth $2$, involving the central biprojection lattice. The dual group case recovers a known result involving the normal subgroup lattice.

% \tableofcontents
 
\section{Basics on lattice theory}
A \emph{lattice} $(L, \wedge , \vee)$ is a poset $L$  in which every two elements $a,b$ have a unique supremum (or \emph{join}) $a \vee b$ and a unique infimum (or \emph{meet}) $a \wedge b$. Let $G$ be a finite group. The set of subgroups $ K \subseteq G$ forms a lattice,  denoted by $\mathcal{L}(G)$, ordered by $\subseteq$, with $K_1 \vee K_2 = \langle K_1,K_2 \rangle$ and $K_1 \wedge K_2 =  K_1 \cap K_2 $. A \emph{sublattice} of $(L, \wedge , \vee)$ is a subset $L' \subseteq L$ such that $(L', \wedge , \vee)$ is also a lattice. If $a,b \in L$ with $a \le b$, then the \emph{interval} $[a,b]$ is the sublattice $\{c \in L \ \vert \ a \le c \le b \}$. Any finite lattice is \emph{bounded}, i.e. admits a minimum and a maximum, denoted by $\hat{0}$ and $\hat{1}$. The \emph{atoms} are the minima of $L \setminus \{\hat{0}\}$. The \emph{coatoms} are the maxima of $L \setminus \{\hat{1}\}$. Consider a finite lattice, $b$ the join of its atoms and $t$ the meet of its coatoms, then let call $[\hat{0},b]$ and $[t,\hat{1}]$ its \emph{bottom} and \emph{top intervals}. A lattice is \emph{distributive} if the join and meet operations distribute over each other.
\noindent  A distributive bounded lattice is called \emph{Boolean} if any element $b$ admits a unique \emph{complement} $b^{\complement}$ (i.e. $b \wedge b^{\complement} = \hat{0}$ and $b \vee b^{\complement} = \hat{1}$). The subset lattice of $\{1,2, \dots, n \}$, with union and intersection, is called the Boolean lattice $\mathcal{B}_n$ of rank $n$. Any finite Boolean lattice is isomorphic to some $\mathcal{B}_n$.

\begin{remark} \label{atomistic}  \label{veeint}
 A finite lattice is Boolean if and only if it is \emph{uniquely atomistic}, i.e. every element can be written uniquely as a join of atoms. It follows that if $[a,b]$ and $[c,d]$ are intervals in a Boolean lattice, then
 $$[a,b] \vee [c,d] := \{ k \vee k' \ | \ k \in [a,b], k' \in [c,d] \},$$ is the interval $[a \vee c, b \vee d]$.
\end{remark} 
   
\noindent We refer to \cite{sta} for more details. 

\section{Subfactor planar algebras and biprojections} \label{sectionB}
For the notions of subfactor, subfactor planar algebra and basic properties, we refer to \cites{js,jo4,sk2}. See also \cite[Section 3]{pa} for a short introduction. Let $ (N \subseteq M) $ be a finite index irreducible subfactor. The $ n $-box spaces $ \mathcal{P}_{n,+} $ and $ \mathcal{P}_{n,-} $ of the planar algebra $ \mathcal{P}=\mathcal{P}(N \subseteq M) $, are $ N' \cap M_{n-1} $ and $ M' \cap M_{n} $. Let $ R(a) $ be the range projection of $ a \in \mathcal{P}_{2,+} $. We define the relations $ a \preceq b $ by $ R(a) \le R(b) $, and $ a \sim b $ by $ R(a) = R(b) $. Let $ e_1:=e^M_N $ and $ \id:=e^M_M $ be the Jones and the identity projections in $ \mathcal{P}_{2,+} $. Let $\tr$ be the normalized trace (i.e. $ \tr(\id) = 1 $). Then $ \tr(e_1) = |M:N|^{-1} = \delta^{-2} $. Let $ \mathcal{F}: \mathcal{P}_{2,\pm} \to \mathcal{P}_{2,\mp} $ be the Fourier transform ($1$-click or $ 90^{\circ} $ rotation) and let $ a * b = \mathcal{F}(\mathcal{F}^{-1}(a) \cdot \mathcal{F}^{-1}(b)) $ be the coproduct of $ a,b \in \mathcal{P}_{2,\pm} $. If $a,b$ are positive then so is $a * b$ by \cite[Theorem 4.1]{li}.
 
\begin{definition}[\cite{li} Def. 2.14] \label{debi}
A \emph{biprojection} is a nonzero projection $ b \in \mathcal{P}_{2,\pm}$ with $ \mathcal{F}(b) $ a multiple of a projection. 
\end{definition}
Note that $ e_1=e^M_N $ and $ \id=e^M_M $ are biprojections.
\begin{theorem}[\cite{bi} p212] \label{bisch}
A projection $ b \in \mathcal{P}_{2,+}$ is a biprojection if and only if it is the Jones projection $ e^M_K $ of an intermediate subfactor $ N \subseteq K \subseteq M $. 
\end{theorem}
Therefore, the set of biprojections is a lattice of the form $ [e_1,\id] $.

\begin{definition} \label{index}  Consider the intermediate subfactors $ N \subseteq K_1 \subseteq K_2 \subseteq M $, and let $ b_i \in [e_1,\id]$ be the biprojection $ e^M_{K_i}$. We define $ \mathcal{P}(b_1, b_2):= \mathcal{P}(K_1 \subseteq K_2) $ and $$ |b_2:b_1|:= \tr(b_2)/ \tr(b_1)=|K_2:K_1|. $$ 
\end{definition}

\begin{definition} \label{gener}
Consider $a \in \mathcal{P}_{2,+}$ positive, and let $p_n$ be the range projection of $\sum_{k=1}^n a^{*k}$. By finiteness, there exists $N$ such that for all $m \ge N$, $p_m = p_N$, which is a biprojection \cite[Lemma 4.14]{li}, denoted $\langle a \rangle$, called the \emph{biprojection generated by $a$}. It is the smallest biprojection $b \succeq a$. For $S$ a finite set of positive elements of $\mathcal{P}_{2,+}$, let  $\langle S \rangle $ be $ \langle \sum_{s \in S}s \rangle$.
\end{definition}

\begin{proposition}[\cite{p1}] \label{mini}
Let $p \in \mathcal{P}_{2,+}$ be a minimal central projection. Then, there exists $u \le p$ minimal projection such that $\langle u \rangle = \langle p \rangle$.
\end{proposition}

\begin{definition}[\cite{p1}]  A planar algebra $\mathcal{P}$ is \emph{weakly cyclic} (or \emph{w-cyclic}) if it satisfies one of the following equivalent assertions:  
\begin{itemize}
\item $\exists u \in  \mathcal{P}_{2,+}$  minimal projection  such that $\langle u \rangle=\id$,
\item $\exists p \in  \mathcal{P}_{2,+}$  minimal central projection  such that $\langle p \rangle=\id$.
\end{itemize}
\end{definition}

\section{Euler totient} \label{sectionE}
We define a notion of Euler totient on the irreducible subfactor planar algebras as an extension of the usual Euler's totient function on the positive integers.

\begin{definition} \label{debius}
The \emph{M\"obius function} $\mu$ of a finite poset $P$ is defined inductively as follows. For $a,b \in P$ with $a \le b$,  
 $$\mu(a,b):= \begin{cases}
            \hspace{.5cm} 1 \hspace{.5cm} \text{if}~ a=b,\\
             - \sum_{c \in (a,b]}\mu(c,b) \hspace{.5cm} \text{otherwise}.
            \end{cases} $$
\end{definition}
The following result can be seen as a Boolean representation for the M\"obius function of a finite lattice.

\begin{theorem}[Crosscut Theorem] \label{crosscut}
Let $L$ be a finite lattice and $a_1 , \dots , a_n$ its coatoms. Consider the (order-reversing) map $m: \mathcal{B}_n \to L$ such that
$$m(I) =  \begin{cases}
            \hspace{.5cm} \hat{1} \hspace{.5cm} \text{if}~ I=\emptyset,\\
             \bigwedge_{i \in I}a_i \hspace{.5cm} \text{otherwise}.
            \end{cases}  $$
Then for any $a \in L$, $$\mu(a,\hat{1}) = \sum_{I \in m^{-1}(\{a\})} (-1)^{|I|}$$
\end{theorem}
\begin{proof}
It is a reformulation of \cite[Corollary 3.9.4]{sta} on $[a,\hat{1}]$ with its coatoms.
\end{proof}

\begin{definition}
Let $\mathcal{P}$ be an irreducible subfactor planar algebra, with biprojection lattice $[e_1,\id]$ and M\"obius function $\mu$. For any $b_1,b_2 \in [e_1,\id]$ with $b_1 \le b_2$, consider 
$$ \varphi(b_1,b_2)  :=  \sum_{b \in [b_1,b_2]} \mu(b,b_2) |b:b_1| $$ The \emph{Euler totient} of $\mathcal{P}$  is $\varphi(\mathcal{P}) := \varphi(e_1,\id)$.
\end{definition}

\begin{remark} \label{rever} Let $\hat{\mathcal{P}}$ be the dual subfactor planar algebra of $\mathcal{P}$. Then
$$ \varphi(\mathcal{P}) =  \sum_{b \in [e_1,\id]} \mu(b,\id) |b:e_1| \  \text{ and } \  \varphi(\hat{\mathcal{P}}) =  \sum_{b \in [e_1,\id]} \mu(e_1,b) |\id:b|, $$
because the biprojection lattice of $\hat{\mathcal{P}}$ is the reversed from that of $\mathcal{P}$.
\end{remark}

\begin{lemma}[\cite{hal}] \label{card}
The Euler totient of a finite group $G$,
$$\varphi(G):=\sum_{H \in \mathcal{L}(G)} \mu(H,G)|H|, $$
is the cardinal of $\{g \in G \ | \   \langle g \rangle = G \}$.
\end{lemma} 
\begin{proof} By Theorem \ref{crosscut} with its map $m:  \mathcal{B}_n \to \mathcal{L}(G)$, 
$$\varphi(G) =  \sum_{H \in \mathcal{L}(G)} \sum_{\alpha \in m^{-1}(\{H\})} (-1)^{|\alpha|}|H| = \sum_{\alpha \in \mathcal{B}_n} (-1)^{|\alpha|} |m(\alpha)|.$$
Then, by the inclusion-exclusion principle, $\varphi(G) = |G \setminus \bigcup_i m(\{i\})|.$
\end{proof}

\begin{corollary} \label{Eulercyclic} A finite group $G$ is cyclic if and only if $ \varphi(G)$ is nonzero.
\end{corollary}
\begin{proof} If $G$ is cyclic then $G = C_n$ for some positive integer $n$, and $\varphi(G)  = \varphi(n) \neq 0$. Conversely, if $ \varphi(G) \neq 0$, then $G$ is cyclic by Lemma \ref{card}.
\end{proof}

Note that $\varphi(\mathcal{P}(R \subseteq R \rtimes G))=\varphi(G)$ and $\varphi(C_n) = \varphi(n)$, the usual Euler's totient function of $n$. Thus, we can see $\mathcal{P} \mapsto \varphi(\mathcal{P})$ as an extension from the positive integers to the irreducible subfactor planar algebras.  The following proposition extends the usual formula that if $n = \prod_i p_i^{n_i}$ is the prime factorization of $n$, then $$\varphi(n) = \prod_ip_i^{n_i-1} \cdot  \prod_i(p_i-1).$$   

\begin{proposition} \label{topinterval}
Let $[t,\id]$ be the top interval of $[e_1,\id]$. Then $$\varphi(\mathcal{P}) = |t:e_1| \cdot \varphi(t,\id).$$
\end{proposition}
\begin{proof}
If $b \not \in [t,\id]$ then $\mu(b,\id)=0$ by Theorem \ref{crosscut} because $m^{-1}(\{b\}) = \emptyset$. Next if $b \in [t,\id]$ then $|b:e_1| = |b:t| \cdot |t:e_1|$, since the index is multiplicative \cite{jo2}.
\end{proof}

\section{Main result} \label{main}
In this section, we generalize one way of Corollary \ref{Eulercyclic} by the following theorem. We will see later with Remark \ref{conv} that the converse is false in general.

\begin{theorem} \label{mainsubfactor}
An irreducible subfactor planar algebra $\mathcal{P}$ is w-cyclic if its Euler totient $\varphi(\mathcal{P})$ is nonzero.
\end{theorem}
\begin{proof} 
Let $p_1, \dots, p_r$ be the minimal central projections of $\mathcal{P}_{2,+}$. Consider the sum $$S(i):= \sum_{b \in [e_1,\id]} \mu(b,\id)   \tr(bp_i).$$
By Theorem \ref{crosscut}  with its map $m:  \mathcal{B}_n \to [e_1,\id]$,
$$ S(i)= \sum_{b \in [e_1,\id]} \sum_{\beta \in m^{-1}(\{b\})} (-1)^{|\beta|}   \tr(bp_i)= \sum_{\beta \in \mathcal{B}_n} (-1)^{|\beta|}   \tr(m(\beta)p_i).$$

\noindent Recall that the map $m$ is order-reversing and the image of the atoms of $\mathcal{B}_n$ are the coatoms of $[e_1,\id]$; let call these latter ones $b_1, \dots , b_n$. Let $A_i$ be the set of atoms $\alpha $ of $ \mathcal{B}_n$ satisfying $p_i \le m(\alpha)$, and  $B_i$ the set of atoms not in $A_i$. Let $\alpha_i$ (resp. $\beta_i$) be the join of all the elements of $A_i$ (resp. $B_i$). 

\begin{claim}
For $\alpha \in \mathcal{B}_n$, $p_i \le m(\alpha) \Leftrightarrow \alpha \in [\alpha_i,\hat{1}]$.
\end{claim}
\begin{claimproof}
Just observe that $p_i \le \bigwedge_{j \in \alpha}b_j$ if and only if $\forall j \in \alpha , \ p_i \le b_j$.
\end{claimproof} \\ Now by Remark \ref{veeint}, we have 
$$\mathcal{B}_n = [\emptyset,\alpha_i] \vee [\emptyset,\beta_i]= \bigsqcup_{\alpha \in [\emptyset,\alpha_i]} \alpha \vee [\emptyset,\beta_i],$$
Consider the following sum
$$T(i):=\sum_{\beta \in [\emptyset,\beta_i]} (-1)^{|\beta|}  \tr(m(\beta)p_i)$$ For any $\alpha \in [\emptyset,\alpha_i]$ and $\beta \in [\emptyset,\beta_i]$, we have $(-1)^{|\alpha \vee \beta|} = (-1)^{|\alpha|}(-1)^{|\beta|}$ and   $$m(\alpha \vee \beta)p_i = m(\alpha)p_i \wedge m(\beta)p_i = m(\beta)p_i.$$ So we get that 
$$S(i)=\sum_{\alpha \in [\emptyset,\alpha_i]} (-1)^{|\alpha|} T(i) = T(i) \cdot (1-1)^{|A_i|}.$$ 

\begin{claim}
The planar algebra $\mathcal{P}$ is w-cyclic if and only if $\exists i$ with $|A_i| = 0$.
\end{claim}
\begin{claimproof}
It is w-cyclic if and only if $\exists i$ with $\langle p_i \rangle = \id$, if and only if for any coatom $b \in [e_1,\id]$, $p_i \not \le b$, if and only if $|A_i| = 0$.
\end{claimproof} 
   
\noindent If $\mathcal{P}$ is not w-cyclic, then $\forall i \ |A_i| \neq 0$, so $S(i) = 0$;
 but $|b:e_1| = \tr(b)/\tr(e_1)$,   $\tr(b) = \sum_i \tr(bp_i)$  and $\tr(e_1) = \delta^{-2}$, so $\varphi(e_1,\id) = \delta^2 \sum_{i=1}^r S(i) = 0$.
\end{proof} 

\noindent It is a double generalization of \cite[Theorem 3.10]{bp} from groups to subfactors and from Boolean lattice to any lattice. It is also a purely combinatorial criterion for a subfactor planar algebra to be w-cyclic. 

\section{Generalization of a problem of K.S. Brown} \label{SectionK} 
We will explain how this path of research conducts to a generalization of a problem of K.S. Brown in finite group theory and to a possible counter-example of the problem of Watatani on representing any finite lattice as a biprojection lattice.  

\begin{proposition} \label{central}
Assume that every biprojection is central. Then $$\varphi(\mathcal{P}) = \delta^2\sum_{\langle p_i \rangle = \id} \tr(p_i)$$
with $p_1, \dots , p_r$ the minimal central projections of $\mathcal{P}_{2,+}$.
\end{proposition}  
\begin{proof}
First of all, $|b:e_1| = \tr(b)/ \tr(e_1)$ and $ \tr(e_1) = \delta^{-2}$, so $$\varphi(\mathcal{P}) = \delta^2\sum_{b \in [e_1,\id]} \mu(b,\id) \tr(b).$$
 By Theorem \ref{crosscut} with its map $m:  \mathcal{B}_n \to [e_1,\id]$, 
$$\varphi(\mathcal{P}) = \delta^2 \sum_{\beta \in \mathcal{B}_n} (-1)^{|\beta|} \tr(m(\beta)),$$ but every biprojection is central, so by the inclusion-exclusion principle
$$ \hspace*{1.8cm} \varphi(\mathcal{P}) = \delta^2 [\tr(\id) - \tr(R[\sum_{i=1}^n m(\{i\})])] = \delta^2\sum_{\langle p_i \rangle = \id} \tr(p_i). \hspace*{1.8cm} \qedhere $$
\end{proof}
\noindent  It follows that the converse of Theorem \ref{mainsubfactor} is true if every biprojection is central; if moreover they form a Boolean lattice, then by \cite[Theorem 4.26]{p1},  $\mathcal{P}$ is w-cyclic, so $\varphi(\mathcal{P})$ is nonzero; and without assuming the biprojections to be central, $\mathcal{P}$ is still w-cyclic by \cite[Theorem 5.9]{p2}, and $\varphi(\mathcal{P})$ is suspected to be still nonzero:
\begin{conjecture} \label{conj2} If $[e_1,\id]$ is Boolean then $\varphi(\mathcal{P})$ is nonzero. 
\end{conjecture} 

\begin{remark} If $[e_1,\id]$ is Boolean of rank $n+1$, we wonder whether $\varphi(\mathcal{P}) \ge\phi^{n}$, with $\phi$ the golden ratio. If this lower bound is correct, then it is optimal as realized by the tensor product $\mathcal{TLJ}(\sqrt{2}) \otimes \mathcal{TLJ}(\phi)^{\otimes n}$, thanks to \cite[Theorem 4.8]{p1}. \end{remark}

By Theorem \ref{mainsubfactor}, a proof of Conjecture \ref{conj2} would be stronger than \cite{p2}.  This conjecture can be seen as the Boolean restriction of the following planar algebraic generalization of a problem of K.S. Brown (considered hard in \cite{sw}):
\begin{conjecture} \label{conj3} Let $\mathcal{P}$ be an irreducible subfactor planar algebra. Then its Euler characteristic (defined below) is nonzero.
$$\chi(\mathcal{P}) := -\sum_{b \in [e_1,\id]} \mu(b,\id)|\id:b|$$
Gasch\"utz proved the usual Brown's problem, i.e. for $\mathcal{P}(R \subseteq R \rtimes G)$, if $G$ is solvable.
\end{conjecture} 

\noindent It is an extension of Conjecture \ref{conj2} because by Remark \ref{rever}, $\chi(\mathcal{P})= \pm \varphi(\hat{\mathcal{P}})$ if $[e_1,\id]$ is Boolean (of rank $n$), because then $\mu(e_1,b) = (-1)^n \mu(b,\id)$. This last equality holds more generally for any Eulerian lattice, which is a graded lattice such that $\mu(a,b)=(-1)^{|b|-|a|}$ for $a \le b$, with $a \mapsto |a|$ the rank function. This leads us to:

\begin{conjecture} \label{conj4} If $[e_1,\id]$ is Eulerian then $\varphi(\mathcal{P})$ is nonzero, and so $\mathcal{P}$ is w-cyclic. 
\end{conjecture}
\noindent The last part of Conjecture \ref{conj4}, deduced from the first together with Theorem \ref{mainsubfactor}, would provide an extension of Ore's theorem \cite[Theorem 5.9]{p2} from top Boolean to top Eulerian lattices (notation: a lattice is called top $X$ if its top interval is $X$). Note that the face lattice of any convex polytope is Eulerian \cite[Proposition 3.8.9]{sta}. 

\begin{proposition} \label{eub} An Eulerian subgroup lattice (of a finite group) is Boolean. 
\end{proposition}
\begin{proof}
If $\mathcal{L}(G)$ is Eulerian then $\mu(1,G)=\pm 1$, but \cite[Th\'eor\`eme 3.1]{kt} states that $$\mu(1,G) \in \frac{|G|}{|G:G'|_0} \mathbb{Z}$$ with $G'$ the commutator subgroup of $G$ and $|G:G'|_0$ the square-free part of $|G:G'|$. Then $|G|= |G:G'|_0$, and so $G'=1$. It follows that $G$ is abelian with $|G|$ square-free, so $G$ is cyclic of square-free order and $\mathcal{L}(G)$ is Boolean.
\end{proof}

\noindent It is unknown whether Proposition \ref{eub} extends to any interval. This leads us to:

\begin{question} Is there an irreducible subfactor planar algebra with a non-Boolean Eulerian biprojection lattice?
\end{question}
\noindent It is unknown whether any finite lattice admits a biprojection lattice representation \cites{wa,xu}, even in the group-subgroup case, so the smallest non-Boolean Eulerian lattice (the face lattice of the square polytope, below) could be a counter-example.
$$ 
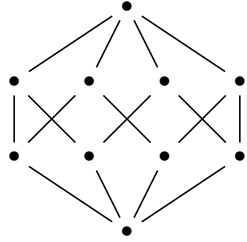
\begin{tikzpicture} \small
\node (A0) at (0,0)     {$\bullet$};

\node (A1) at (-1.5,-1) {$\bullet$};
\node (A2) at (-0.5,-1) {$\bullet$};
\node (A3) at (0.5,-1)  {$\bullet$};
\node (A4) at (1.5,-1)  {$\bullet$};

\node (A5) at (-1.5,-2) {$\bullet$};
\node (A6) at (-0.5,-2) {$\bullet$};
\node (A7) at (0.5,-2)  {$\bullet$};
\node (A8) at (1.5,-2)  {$\bullet$};

\node (A9) at (0,-3)    {$\bullet$};
 \tikzstyle{segm}=[-,>=latex, semithick]
\draw [segm] (A0)to(A1); \draw [segm] (A0)to(A2); \draw [segm] (A0)to(A3); \draw [segm] (A0)to(A4);

\draw [segm] (A1)to(A5); \draw [segm] (A1)to(A6);
\draw [segm] (A2)to(A5); \draw [segm] (A2)to(A7);
\draw [segm] (A3)to(A6); \draw [segm] (A3)to(A8);
\draw [segm] (A4)to(A7); \draw [segm] (A4)to(A8);

\draw [segm] (A5)to(A9); \draw [segm] (A6)to(A9); \draw [segm] (A7)to(A9); \draw [segm] (A8)to(A9);

\end{tikzpicture} $$
\section{Applications}

We will deduce a non-trivial upper bound from Theorem \ref{mainsubfactor} providing a link between combinatorics and representations in finite group theory by translating.

\begin{definition} The \emph{Euler totient} of an interval of finite groups $[H,G]$ is $$ \varphi(H,G) := \sum_{K \in [H,G]}\mu(K,G) |K:H|. $$
\end{definition} 
\noindent Note that $\varphi(\mathcal{P}(R \rtimes H \subseteq R \rtimes G)) = \varphi(H,G)$.
\begin{corollary}  \label{Eulerthm} There is $g \in G$ with $\langle Hg \rangle = G$ if and only if $ \varphi(H,G)$ is nonzero.
\end{corollary}
\begin{proof} By Proposition \ref{central}; or directly, let 
$M_1, \dots, M_n $ be the coatoms of $[H,G]$. Then, as for the proof of Lemma \ref{card}, $$|H| \varphi(H,G) = |G \setminus \bigcup M_i|,$$ so that  $\varphi(H,G)$ is the cardinal of $\{ Hg \ | \ g \in G \text{ and } \langle Hg \rangle = G \}$.
\end{proof}

\begin{corollary} \label{gene}
The minimal cardinal for a generating set of a finite group $G$, is the minimal length $\ell$ for an ordered chain of subgroups $$\{e\}=H_0 < H_1 < \dots < H_{\ell} = G$$ such that $\varphi(H_i, H_{i+1})$ is nonzero. 
\end{corollary}
\begin{proof} A subset $\{g_1, \dots , g_n \}$ is generating iff there is an ordered chain of subgroups $$\{e\}=K_0 < K_1 < \dots < K_n = G$$ with $K_{i+1}=\langle K_i, g_{i+1} \rangle = \langle K_i g_{i+1} \rangle$, iff $\varphi(K_i, K_{i+1})$ is nonzero by Corollary \ref{Eulerthm}. \end{proof}

\noindent So, the minimal cardinal for a generating set of a finite group $G$ depends only on the subgroup lattice and indices (known to \cite{hal}).  We will generalize Corollary \ref{gene} by a non-trivial upper bound. Let $\mathcal{P}$ be an irreducible subfactor planar algebra.

\begin{lemma}[\cite{p1}] \label{lem}  Let $b_1 < b_2$ be biprojections. If $\mathcal{P}(b_1,b_2)$ is w-cyclic, then there is a
minimal projection $u \in \mathcal{P}_{2,+}$ such that $b_2 = \langle b_1, u \rangle$.
\end{lemma}

\begin{theorem} \label{upper}
The minimal number $r$ of minimal projections generating the identity biprojection (i.e. $\langle u_1, \dots, u_r \rangle = \id$) is at most the minimal length $\ell$ for an ordered chain of biprojections $$e_1=b_0 < b_1 < \dots < b_{\ell} = \id$$ such that $\varphi(b_i,b_{i+1})$ is nonzero.
\end{theorem}
\begin{proof} Consider a chain as above. By Theorem \ref{mainsubfactor}, $\mathcal{P}(b_i,b_{i+1})$ is w-cyclic, so by Lemma \ref{lem}, there is a minimal projection $u_{i+1} \in \mathcal{P}_{2,+}$ such that $b_{i+1} = \langle b_i, u_{i+1} \rangle$. It follows that $\langle u_1, \dots, u_{\ell} \rangle = \id$.
\end{proof}

We deduce weak dual versions of Corollaries \ref{Eulerthm}, \ref{Eulercyclic} and \ref{gene}, giving the link between combinatorics and representation theory:

\begin{definition} The \emph{dual Euler totient} of the interval $[H,G]$ is $$\hat{\varphi}(H,G) := \sum_{K \in [H,G]}\mu(H,K) |G:K|.$$
\end{definition}  
\noindent  Note that $ \varphi(\mathcal{P}(R^G  \subseteq R^H)) = \hat{\varphi}(H,G)$.

\begin{definition} \label{fixstab} Let $ W $ be a representation of a group $ G $, $ K $ a subgroup of $ G $, and $ X $ a subspace of $ W $. Let the \emph{fixed-point subspace} be $$ W^{K}:=\{w \in W \ \vert \ kw=w \ , \forall k \in K \} $$ and the \emph{pointwise stabilizer subgroup} be $$ G_{(X)}:=\{ g \in G \ \vert \ gx=x \ , \forall x \in X \} $$ \end{definition} 

\begin{definition} \label{linprim}
An interval of finite group $ [H,G] $ is said to be \emph{linearly primitive} if there is an irreducible complex representation $ V $ of $ G $ with $ G_{(V^H)} = H $.
\end{definition}

\begin{theorem}[\cite{p1}] \label{wgrp}
The planar algebra $ \mathcal{P}(R^G \subseteq R^H) $ is w-cyclic if and only if the interval of finite groups $ [H,G] $ is linearly primitive. \end{theorem}

\begin{corollary} \label{dualEulerthm}  If the dual Euler totient $\hat{\varphi}(H,G)$ is nonzero, then the interval of finite groups $ [H,G] $ is linearly primitive.
\end{corollary}
\begin{proof} Apply Theorem \ref{mainsubfactor} to $\mathcal{P}(R^G  \subseteq R^H)$, and then Theorem \ref{wgrp}.
\end{proof}
\noindent In particular, let the dual Euler totient of $G$ be $\hat{\varphi}(G):=\hat{\varphi}(1,G)$. Then:
\begin{corollary} \label{coroirr}  A finite group $G$ admits a faithful irreducible complex representation if its dual Euler totient $\hat{\varphi}(G)$ is nonzero.
\end{corollary}
\begin{proof} It follows from Corollary \ref{dualEulerthm} because $G_{(V)} = \ker(\pi_V)$.
\end{proof}
\noindent It is a purely combinatorial criterion for a finite group to have an irreducible faithful complex representation.
\begin{remark} \label{conv}  The converse is false. The modular maximal-cyclic group $ M_4(2)$ (of order $16$) has a faithful irreducible complex representation, whereas $\hat{\varphi}(M_4(2)) = 0$. This is not surprising because according to Proposition \ref{topinterval}, $\varphi(e_1,\id) \neq 0$ if and only if  $\varphi(t,\id) \neq 0$ with $[t,\id]$ the top interval of $[e_1,\id]$; and the bottom interval of $[1,M_4(2)]$ is $[1,C_2^2]$. Even if we assume that $[1,G]$ is its own bottom interval, the converse is still false: there are exactly two counter-examples of index $\le 100$, given by $G=D_8 \rtimes C_2^2$ or $D_8 \rtimes S_3$ (of order $64$ and $96$ respectively).
\end{remark}
\begin{theorem}[\cite{bu} \S 226] \label{faith}
A complex representation $V$ of a finite group $G$ is faithful if and only if for any irreducible complex representation $W$ there is an integer $n$ such that $W \preceq V^{\otimes n}$.
\end{theorem}
\begin{corollary} \label{gened}
The minimal number of irreducible components for a faithful complex representation of $G$, is at most the minimal length $\ell$ for an ordered chain of subgroups $$\{e\}=H_0 < H_1 < \dots < H_{\ell} = G$$ such that $\hat{\varphi}(H_i, H_{i+1})$ is nonzero.
\end{corollary}
\begin{proof} Consider a chain as above. By Corollary \ref{upper} applied to $\mathcal{P}(R^G  \subseteq R)$, we have $\langle u_1, \dots, u_{\ell} \rangle = \id$, with $u_i$ minimal projection. Let $p_i$ be the central support of $u_i$. Then $\langle p_1+ \dots + p_{\ell} \rangle = \id$. But the coproduct of two minimal central projections is given by the tensor product of the associated irreducible representations of $G$ \cite[Corollary 7.5]{p2}. So by Definition \ref{gener} and Theorem \ref{faith}, the representation $V_1 \oplus \cdots \oplus V_{\ell}$ (with $V_i$ the irreducible complex representation $\im(p_i)$) is faithful. \end{proof}

\noindent Note that Corollary \ref{gened} extends to any finite dimensional Kac algebra as for \cite[Remark 6.14]{p2}. It can also be improved by taking for $H_0$ any core-free subgroup of $H_1$ (instead of just $\{e\}$), thanks to \cite[Lemma 6.13]{p2}. In particular:
\begin{corollary} \label{core} A finite group $G$ admits a faithful irreducible complex representation if there is a core-free subgroup $H<G$ with $\hat{\varphi}(H,G)$ nonzero.
\end{corollary}
\noindent This criterion is more efficient than Corollary \ref{coroirr} (consider for example $G$ simple), but it is no more purely combinatorial.
\begin{question} Is the converse of Corollary \ref{core} true?
\end{question}

\section{Additional result for the depth $2$}

We can prove an additional result in the depth $2$ case (involving the central biprojection lattice) coming from the fact that the irreducible depth $2$ subfactor planar algebras correspond to the Kac algebras \cite{kls}. Note that the group subfactors are depth $2$, since the group algebras are examples of Kac algebras.

\begin{theorem}[Splitting, \cite{kls} p39] \label{split} Let $\mathcal{P}$ be an irreducible depth $2$ subfactor planar algebra. Any element $x \in \mathcal{P}_{2,+} $ splits as follows:
$$\begin{tikzpicture}[scale=.5, PAdefn]
	\clip (0,0) circle (2.6cm);
	\draw[shaded] (-0.15,0) -- (100:4cm) -- (80:4cm) -- (0.15,0);
	\draw[shaded] (0,0) -- (-120:4cm) -- (-60:4cm) -- (0,0);
	\node at (0,0) [Tbox, inner sep=1.5mm] {{$x$} };
	%\node at (180:1.3cm) {$\star$};
 \draw[fill=white] (-0.7,-4) .. controls ++(120:3cm) and ++(60:3cm) .. (0.7,-4);	
	%\draw[very thick] (0,0) circle (3cm);
\end{tikzpicture}  
\hspace{-0.4cm} =  \  	
	\begin{tikzpicture}[scale=.5, PAdefn]
	\clip (0,0) circle (2.6cm);
	\draw[shaded] (-2,0) -- (0,5) -- (2,0) -- (-2,0);
	\draw[fill=white] (-0.7,-0.5) .. controls ++(120:3cm) and ++(60:3cm) .. (0.7,-0.5);
	%\draw[shaded] (0,1) -- (-120:4cm) -- (-60:4cm) -- (0,1);
	\draw[shaded] (-1.6,0) -- (-120:4cm) -- (-100:4cm) -- (-1.4,0);
	\draw[shaded] (1.4,0) -- (-80:4cm) -- (-60:4cm) -- (1.6,0);
	\node at (-1.5,0) [Tbox, inner sep=0.2mm] {{$x_{(1)}$} };
	\node at (1.5,0) [Tbox, inner sep=0.2mm] {{$x_{(2)}$} };
	%\node at (180:1.3cm) {$\star$};
	%\draw[very thick] (0,0) circle (3cm);
\end{tikzpicture} 
\ \text{ and}  
\begin{tikzpicture}[scale=.5, PAdefn]
	\clip (0,0) circle (2.6cm);
	\draw[shaded] (-0.15,0) -- (260:4cm) -- (280:4cm) -- (0.15,0);
	\draw[shaded] (0,0) -- (60:4cm) -- (120:4cm) -- (0,0);
	\node at (0,0) [Tbox, inner sep=1.5mm] {{$x$} };
	%\node at (0:1.3cm) {$\star$};
 \draw[fill=white] (-0.7,4) .. controls ++(240:3cm) and ++(300:3cm) .. (0.7,4);	
	%\draw[very thick] (0,0) circle (3cm);
\end{tikzpicture} 
\hspace{-0.4cm} =  \ 	
\begin{tikzpicture}[scale=.5, PAdefn]
	\clip (0,0) circle (2.6cm);
	\draw[shaded] (-2,0) -- (0,-5) -- (2,0) -- (-2,0);
	\draw[fill=white] (-0.7,0.5) .. controls ++(240:3cm) and ++(300:3cm) .. (0.7,0.5);
	%\draw[shaded] (0,1) -- (-280:4cm) -- (-300:4cm) -- (0,1);
	\draw[shaded] (-1.6,0) -- (-240:4cm) -- (-260:4cm) -- (-1.4,0);
	\draw[shaded] (1.4,0) -- (80:4cm) -- (60:4cm) -- (1.6,0);
	\node at (-1.5,0) [Tbox, inner sep=0.2mm] {{$x_{(1)}$} };
	\node at (1.5,0) [Tbox, inner sep=0.2mm] {{$x_{(2)}$} };
	%\node at (180:1.3cm) {$\star$};
	%\draw[very thick] (0,0) circle (3cm);
\end{tikzpicture}$$  

\noindent Note that $\Delta(x) = x_{(1)} \otimes x_{(2)}$ is the sumless Sweedler notation for the comultiplication of the corresponding Kac algebra.  
\end{theorem}

\begin{corollary} \label{thmZZ}
If $a,b \in \mathcal{P}_{2,+}$ are central, then so is the coproduct $a * b$.
\end{corollary}
\begin{proof}
 This diagrammatic proof by splitting is due to Vijay Kodiyalam.
$$ \hspace*{0.2cm} (a*b) \cdot x =  \hspace{-0.3cm}
\begin{tikzpicture}[scale=.5, PAdefn]
	\clip (0,0) circle (3cm);
	\draw[shaded] (-0.15,0) -- (100:4cm) -- (80:4cm) -- (0.15,0);
	\draw[shaded] (0,1.2) -- (-130:4cm) -- (-50:4cm) -- (0,1.2);
	\draw[shaded] (0,1.2) -- (-130:4cm) -- (-50:4cm) -- (0,1.2);
	\node at (0,1.2) [Tbox, inner sep=1.35mm] {{$x$} };
	\draw[fill=white] (0,-1.2) circle (0.75cm);
	\node at (-1.2,-1.2) [Tbox, inner sep=1.5mm] {{$a$} };
	\node at (1.2,-1.2) [Tbox, inner sep=1.35mm] {{$b$} };
 %\draw[fill=white] (-0.7,-4) .. controls ++(120:3cm) and ++(60:3cm) .. (0.7,-4);	
\end{tikzpicture}  
\hspace{-0.4cm} =     \textbf{\hspace{-0.2cm}}	
	\begin{tikzpicture}[scale=.5, PAdefn]
	\clip (0,0) circle (3cm);
	\draw[shaded] (-2,1) -- (0,5) -- (2,1) -- (-2,1);
	\draw[shaded] (-2,-1) -- (0,-5) -- (2,-1) -- (-2,-1);
	\draw[shaded] (-1.65,-1.2) -- (1.65,-1.2) -- (1.65,1.4) -- (-1.65,1.4) -- (-1.65,-1.2);
	\draw[fill=white] (0,0) ellipse (0.8cm and 2.3cm);
	\node at (-1.2,-1.2) [Tbox, inner sep=1.5mm] {{$a$} };
	\node at (1.2,-1.2) [Tbox, inner sep=1.35mm] {{$b$} };
	\node at (-1.2,1.4) [Tbox, inner sep=0.1mm] {{$x_{(1)}$} };
	\node at (1.2,1.4) [Tbox, inner sep=0.1mm] {{$x_{(2)}$} };
	%\node at (180:1.3cm) {$\star$};
\end{tikzpicture} 
\hspace{-0.1cm} =    \hspace{-0.2cm}	
	\begin{tikzpicture}[scale=.5, PAdefn]
	\clip (0,0) circle (3cm);
	\draw[shaded] (-2,1) -- (0,5) -- (2,1) -- (-2,1);
	\draw[shaded] (-2,-1) -- (0,-5) -- (2,-1) -- (-2,-1);
	\draw[shaded] (-1.65,-1.2) -- (1.65,-1.2) -- (1.65,1.4) -- (-1.65,1.4) -- (-1.65,-1.2);
	\draw[fill=white] (0,0) ellipse (0.8cm and 2.3cm);
	\node at (-1.2,-1.2) [Tbox, inner sep=0.1mm] {{$x_{(1)}$} };
	\node at (1.2,-1.2) [Tbox, inner sep=0.1mm] {{$x_{(2)}$} };
	\node at (-1.2,1.4) [Tbox, inner sep=1.5mm] {{$a$} };
	\node at (1.2,1.4) [Tbox, inner sep=1.35mm] {{$b$} };
	%\node at (180:1.3cm) {$\star$};
\end{tikzpicture} 
  \hspace{-0.05cm} =  \hspace{-0.3cm}
\begin{tikzpicture}[scale=.5, PAdefn]
	\clip (0,0.2) circle (3cm);
	\draw[shaded] (-0.15,2) -- (-80:4.2cm) -- (-100:4.2cm) -- (0.15,2);
	\draw[shaded] (0,-1) -- (50:5cm) -- (130:5cm) -- (0,-1);
	\node at (0,-1) [Tbox, inner sep=1.35mm] {{$x$} };
	\draw[fill=white] (0,1.4) circle (0.75cm);
	\node at (-1.2,1.4) [Tbox, inner sep=1.5mm] {{$a$} };
	\node at (1.2,1.4) [Tbox, inner sep=1.35mm] {{$b$} };
 %\draw[fill=white] (-0.7,-4) .. controls ++(120:3cm) and ++(60:3cm) .. (0.7,-4);	
\end{tikzpicture}  
\hspace{-0.3cm} = x \cdot (a*b) \hspace*{0.2cm} \qedhere $$
 \end{proof}
 
 \begin{corollary}
 The set of central biprojections is a sublattice of $[e_1, \id]$.
 \end{corollary}  
 \begin{proof}
 Let $b_1$, $b_2$ be central biprojections. Then $b_1 \wedge b_2$ is central, and $b_1 \vee b_2$ is the range projection of $(b_1 * b_2)^{*k}$ for $k$ large enough, so is central by Corollary \ref{thmZZ}.
 \end{proof}

Let $\mathcal{C}$ be the  central biprojection lattice and $\mu_{\mathcal{C}}$ its M\"obius function. Let the central Euler totient of $\mathcal{P}$ be $$\varphi_{\mathcal{C}}(\mathcal{P}):=\sum_{b \in \mathcal{C}} \mu_{\mathcal{C}}(b,\id)|b:e_1|.$$
\begin{theorem} \label{cent} Let $p_1, \dots , p_r$ be the minimal central projections of $\mathcal{P}_{2,+}$. Then
$$\varphi_{\mathcal{C}}(\mathcal{P}) = \delta^2\sum_{\langle p_i \rangle = \id} \tr(p_i).$$
\end{theorem}
\begin{proof} By Corollary \ref{thmZZ}, a central projection generates a central biprojection. The rest of the proof is identical to that of Proposition \ref{central}.
\end{proof}

\begin{corollary} \label{depth2} Let $\mathcal{P}$ be an irreducible subfactor planar algebra of depth $2$. Then $\mathcal{P}$ is w-cyclic if and only if $\varphi_{\mathcal{C}}(\mathcal{P})$ is nonzero.
\end{corollary} 
\begin{remark} The last three results extend to any irreducible subfactor planar algebra in which any central projection generates a central biprojection.
\end{remark}
 The rest is only translation. Let $G$ be a finite group, $\mathcal{N}(G)$ its normal subgroup lattice and $\mu_{\mathcal{N}}$ the M\"obius function of $\mathcal{N}(G)$. Let the dual normal Euler totient be
$$ \hat{\varphi}_{\mathcal{N}}(G) := \sum_{H \in \mathcal{N} (G)}\mu_{\mathcal{N}}(1,H) |G:H| $$
\begin{corollary} A finite group $G$ has a faithful irreducible complex representation if and only if $\hat{\varphi}_{\mathcal{N}}(G)$ is nonzero.
\end{corollary}
\begin{proof}
Apply  Corollary \ref{depth2} on $\mathcal{P}(R^G \subseteq R)$; or Theorem \ref{cent}  with $V_i=\im(p_i)$ then
$$ \hspace*{3.87cm} \hat{\varphi}_{\mathcal{N}}(G) = \sum_{V_i \text{ faithful}} \dim(V_i)^2. \hspace*{3.87cm} \qedhere $$
\end{proof}
\noindent The above equality can be proved directly from the content of the page 97 in \cite{palfy}.

\section{Acknowledgments} 
The author would like to thank the Isaac  Newton  Institute  for  Mathematical  Sciences,  Cambridge,  for  support and hospitality  during  the  programme  \textit{Operator  Algebras:  Subfactors  and  their  applications}, where work on this paper was undertaken. This work was supported by EPSRC grant no EP/K032208/1; thanks to David Evans for the invitation; an anonymous referee for her/his interest in this work and useful suggestions; Richard Stanley for showing the crosscut theorem; John Shareshian for Kratzer-Th\'evenaz theorem and Eulerian lattices; Benjamin Steinberg and Frieder Ladisch for P\'alfy's formula.

\end{document}